\documentclass[a4paper,12pt]{amsart}
\usepackage{verbatim}
\usepackage{mathtools}
\usepackage{amsmath}
\usepackage{enumitem}
\usepackage{cite}
\usepackage{hyperref}
\newtheorem{theorem}{Theorem}[section]
\newtheorem*{theorem*}{Theorem}
\newtheorem{lemma}[theorem]{Lemma}
\newtheorem{proposition}[theorem]{Proposition}
\newtheorem{corollary}[theorem]{Corollary}

\newtheorem*{question1}{Question 1}
\newtheorem*{question2}{Question 2}

\theoremstyle{definition}
\newtheorem{definition}[theorem]{Definition}

\theoremstyle{remark}

\def\N{{\mathbb
		N}}
\def\R{{\mathbb R}}

\DeclareMathOperator{\co}{conv}
\DeclareMathOperator{\lin}{lin}

\newcommand{\xast}{x^{\ast}}
\newcommand{\xastast}{x^{\ast\ast}}
\newcommand{\yast}{y^{\ast}}
\newcommand{\yastast}{y^{\ast\ast}}
\newcommand{\zast}{z^{\ast}}
\newcommand{\Xast}{X^{\ast}}
\newcommand{\Xastast}{X^{\ast\ast}}
\newcommand{\Yastast}{Y^{\ast\ast}}

\title[Bidual octahedral renormings and strong regularity]
{Bidual octahedral renormings and strong regularity in Banach spaces} 

\begin{document}
\begin{abstract} We prove that every separable Banach space containing an isomorphic copy of $\ell_1$ can be equivalently renormed so that the new bidual norm is octahedral. This answers, in the separable case, a question in Godefroy (1989) \cite{godefroy_metric_1989}. As a direct consequence, we obtain that every dual Banach space, with a separable predual and failing to be strongly regular, can be equivalently renormed with a dual norm to satisfy the strong diameter two property. 

\end{abstract}

 \author{ Johann Langemets }\thanks{The work of J. Langemets was supported by the Estonian Research Council grant (PUTJD702), by institutional research funding IUT (IUT20-57) of the Estonian Ministry of Education and Research, and by a grant of the Institute of Mathematics of the University of Granada (IEMath-GR)}
\address{Institute of Mathematics and Statistics, University of Tartu, J. Liivi 2, 50409 Tartu, Estonia}\email{johann.langemets@ut.ee}

 \author{ Gin\'es L\'opez-P\'erez }\thanks{The work of G. L\'opez-P\'erez was supported by MICINN (Spain) Grant PGC2018-093794-B-I00 and by Junta de Andaluc\'ia Grant FQM-0185.}
\address{Universidad de Granada, Facultad de Ciencias.
Departamento de An\'{a}lisis Matem\'{a}tico, 18071-Granada
(Spain)}\email{ glopezp@ugr.es}



\maketitle

\section{Introduction}

The existence of isomorphic copies of $\ell_1$ in a Banach space has been a central topic in Banach space theory, in fact there are well-known isomorphic characterizations independent of the considered norm in the space, as the one given by H. Rosenthal in \cite{rosenthal}. Also, there are purely geometrical characterizations, in terms of the considered norm in the space. For example, in \cite{maurey},  B. Maurey shows a celebrated characterization of separable Banach spaces containing isomorphic copies of $\ell_1$ as those separable Banach spaces $X$ satisfying that there exists a  $\xastast\in \Xastast\setminus\{0\}$ such that
\[
\Vert x+\xastast \Vert = \Vert x-\xastast \Vert \quad \text{for all $x\in X$}.
\]
Note that this characterization remains true even for equivalent norms. The proof of Maurey's theorem is quite involved, however an elementary approach, which
works in the important special case of subspaces of weakly sequentially complete Banach lattices, is contained in \cite{yagoub-zidi}. Maurey's theorem is closely related to the study of $\ell_1$-types (also studied by V. Kadets, V. Shepelska and D. Werner in \cite{kadets}) and fails in the non-separable case \cite{maurey}. Later, octahedral norms were introduced in an unpublished paper by G. Godefroy and B. Maurey \cite{godefroy_maurey} (see \cite{godefroy_metric_1989}), and this kind of norms was used by G. Godefroy and N. Kalton in \cite{godefroy_kalton} around the study of the ball topology in Banach spaces, with applications to the existence of unique preduals. 

We recall that the norm $\|\cdot\|$ on a Banach space $X$ (or $X$) is called \emph{octahedral}
if, for every finite-dimensional subspace $E$ of $X$ and every $\varepsilon>\nobreak0$,
there is~a $y\in S_X$ such that
\[
\|x+y\|\geq(1-\varepsilon)\bigl(\|x\|+\|y\|\bigr)\quad\text{for all $x\in E$.}
\]
In \cite[Lemma~9.1]{godefroy_kalton} it is proved that a separable Banach space $X$ is octahedral if, and only if, there exists a  $\xastast\in \Xastast\setminus\{0\}$ such that
\[
\Vert x+\xastast \Vert = \Vert x\Vert +\Vert \xastast \Vert \quad \text{for all $x\in X$}.
\]
It is clear that $\ell_1$ (as $L_1$ or $C([0,1])$) is octahedral, and also it is easy to find non-octahedral equivalent norms in $\ell_1$. However, if we allow the space to be renormed, then G. Godefroy proved the following general characterization:

\begin{theorem*}[see {\cite[Theorem~II.4]{godefroy_metric_1989}}]\label{thm: godefroy contain ell1} Let $X$ be a Banach space. The following assertions are equivalent:
	\begin{enumerate}
		\item[(i)] $X$ contains a subspace isomorphic to $\ell_1$.
		\item[(ii)] there exists an equivalent norm $|\cdot|$ in $X$such that $(X,|\cdot|)$ is octahedral.
		\item[(iii)] there exists an equivalent norm $|||\cdot|||$ in $X$ and $\xastast\in \Xastast\setminus\{0\}$ such that 
		\[
		|||x+\xastast|||=|||x|||+|||\xastast||| \quad \text{for all $x\in X$.}.
		\]
	\end{enumerate}
\end{theorem*}

It is known that a Banach space $X$ is octahedral whenever $\Xastast$ is. However, the converse is not true in general. The natural norm of $C([0,1])$ is octahedral, but its bidual norm is not octahedral, because the characteristic function of a singleton is a point of Fr\'{e}chet smoothness of the
bidual norm to the natural norm of $C([0,1])$. Therefore a natural problem was posed by G. Godefroy in \cite{godefroy_metric_1989}:

\begin{question1}[see {\cite[p. 12]{godefroy_metric_1989}}]\label{question: question1}
	If $X$ contains an isomorphic copy of $\ell_1$, does there always exist an equivalent norm on $X$ such that the bidual $\Xastast$ is octahedral? 
\end{question1}

The main goal of this note is to give a positive answer, in the separable case, to the above question. 

Currently it is known that there is a close relationship between octahedral norms and other geometrical properties in Banach spaces, for example the Daugavet property \cite{kadets2} or the diameter two properties. A Banach space $X$ has the strong diameter two property if every convex combination of slices in $B_X$, the unit ball of $X$, has diameter two. It is known that a dual Banach space $X^*$ is octahedral if, and only if, $X$ satisfies the strong diameter two property \cite{becerra_octahedral_2014}, which gives a complete duality relation between these two properties.

Recall that a Banach space is said to be strongly regular if every closed, bounded, and convex subset of $X$ contains convex combinations of slices with diameter arbitrarily small. Strong regularity is a weaker isomorphic property than the well-known Radon-Nikodym property in Banach spaces (see \cite{ghoussub_some_1987} for background). It is known that $\Xast$ is strongly regular if and only if $X$ does not contain isomorphic copies of $\ell_1$ \cite[Corollary~VI.18]{ghoussub_some_1987}. Hence, Question~\ref{question: question1} is a particular case (a dual case in fact) of the general open question: 
\begin{question2}[see {\cite{becerra_extreme_2015}}]\label{question: question2}
	Can every Banach space failing to be strongly regular be equivalently renormed such that it has the strong diameter two property? 
\end{question2}

Let us now describe the organization of the paper. After some notation and preliminaries, we start Section 2 with the definition of an octahedral set in a Banach space, which tries to be a localization of octahedrality for subsets in a Banach space, giving local octahedrality properties in the bidual. Such a set can be obtained in a dual Banach space, whenever the predual satisfies the strong diameter two property, exploiting the dual relation between octahedrality and the strong diameter two property. In Proposition \ref{prop: abstract bidual renorming} we give a general result which produces an octahedral bidual norm in a Banach space from another Banach space with an octahedral set. Section 3 studies the properties of a subset with the strong diameter two property constructed by M. Talagrand \cite{schachermayer_moduli_1989} in the dual of $C(\Delta)$, the space of continuous functions on the Cantor set $\Delta$, which will be crucial to get our main goal. The initial interest of this set was to answer affirmatively to the question posed by J. Diestel and J. Uhl, about the relation between $w^*$-dentability of $w^*$-compact subsets of a dual Banach space and the existence of $\ell_1$-copies in the predual. Finally, we conclude in Lemmas \ref{lemma: f_p is a good set} and \ref{lemma: isormphic copy of ell_1} the existence of an isometric $\ell_1$-sequence in $C(\Delta)$ which is an octahedral set in some subspace of $C(\Delta)$ equipped with a different norm. Section 4 implements the general results of Section 1 for the space $C(\Delta)$ and the octahedral set obtained in Section 3 to get in Theorem \ref{thm: renorming C(Delta)*} an equivalent norm in $C(\Delta)$ with octahedral bidual norm. Finally, using the good embedding of $C(\Delta)^*$ in the dual of every separable Banach space with $\ell_1$-copies, we get at Theorem \ref{thm: separable bidual octahedral} that every separable Banach space with isomorphic copies of $\ell_1$ has an equivalent norm with octahedral bidual norm, which answers Question \ref{question: question1} in the separable case. Also, some partial answer in the non-separable setting will be obtained.  As a direct consequence, we obtain in Corollary \ref{corollary: strongregularity} that every dual Banach space, with a separable predual and failing to be strongly regular, can be equivalently renormed with a dual norm to satisfy the strong diameter two property, which is a partial answer to Question 2  (in fact, an answer to the dual case). We finish with other consequences around the ball topology on Banach spaces and questions.

We pass now to introduce some notation. We consider only real Banach spaces. For a Banach space $X$, $\Xast$ denotes the topological dual of $X$, $B_X$ and $S_X$ stand for the closed unit ball and unit sphere of $X$, respectively, and $w$, respectively $w^*$, denotes the weak and weak-star topology in $X$, respectively $\Xast$. For a subspace $Y$ of $X$, $Y^{\perp}:=\{f\in X^*: f(Y)=\{0\}\}$, which is a subspace of $X^*$. Then $Y^{\perp \perp}$ (the perp of $Y^{\perp}$) is a subspace of $X^{**}$. By $\lin A$ we denote the linear span of the subset $A$ of $X$. A slice of a set $C$ in $X$ is a set of $X$ given by 
\[
S(C,\xast,\alpha):=\{x\in C\colon \xast(x)>\sup \xast(C)-\alpha \},
\]
where $\xast \in S_\Xast$ and $\alpha>0$. A $w^*$-slice of a set $C$ of $\Xast$ is a slice of $C$ determined by elements of $X$. By $j$ we will denote the canonical embedding of $X$ into $\Xastast$. If $Y$ is a subspace of $\Xast$, then $X_{\mid Y}$ will denote the set of functionals in $j(X)$ restricted to $Y$, that is, $X_{\mid Y}:=\{x_{\mid Y}:x\in X\subset X^{**}\}$.

\section{Previous results}
From \cite[Proposition~2.1]{haller} we know that a Banach space $X$ is octahedral if, and only if, for every $n\in \N$, $\varepsilon>0$, and for every  $x_1,\ldots ,x_n\in S_X$ there is a $y\in S_X$ such that $$\Vert x_i + y\Vert \geq 2-\varepsilon\ \quad \text{for every $i\in\{1,\dots,n \}.$}$$
It would be natural to call then a subset $A$ of $S_X$ to be  $X$-octahedral, if it satisfies that for every $n\in \N$, $\varepsilon>0$, and for every  $x_1,\ldots ,x_n\in S_X$ there is $a\in A$ such that $$\Vert x_i + a\Vert \geq 2-\varepsilon \ \quad \text{for every $i\in\{1,\dots,n \}.$}$$
In the case $X$ is separable and octahedral, it is essentially proved in \cite{kadets}, using $\ell_1$-types techniques, that it is possible to choose a $X$-octahedral subset as a $\ell_1$-sequence.
 
In the case $X^{**}$ is octahedral, then the $w^*$-lower semicontinuity of the norm in  $X^{**}$ gives that it is possible to choose a $X^{**}$-octahedral set $A$ in such way that $A\subset X$. The next easy lemma shows different ways to get the octahedrality of $X^{**}$.

\begin{lemma}\label{lemma: bidual octahedral} Let $X$ be a Banach space. The following are equivalent:
	\begin{itemize}
		\item[(i)] $\Xastast$ is octahedral.
		\item[(ii)] for every $n\in \N$, $\xastast_1,\ldots ,\xastast_n\in S_{\Xastast}$, and $\varepsilon>0$, there is $\yastast\in S_{\Xastast}$ such that $$\Vert \xastast_i + \yastast\Vert \geq 2-\varepsilon\ \quad \text{for every $i\in\{1,\dots,n \}.$}$$
		\item[(ii')] for every $n\in \N$, $\xastast_1,\ldots ,\xastast_n\in S_{\Xastast}$, and $\varepsilon>0$ there is $y\in S_{X}$ such that $$\Vert \xastast_i + y\Vert \geq 2-\varepsilon\ \quad \text{for every $i\in\{1,\dots,n \}.$}$$
		\item[(iii)] for every $r\geq1$, $n\in \N$, $\xastast_1,\ldots ,\xastast_n\in rB_{\Xastast}$, and $\varepsilon>0$ there is $\yastast\in B_{\Xastast}\setminus\{0\}$ such that $$\Vert \xastast_i + \yastast\Vert \geq (1-\varepsilon)(\Vert \xastast_i\Vert +1)\ \quad \text{for every $i\in\{1,\dots,n \}.$}$$
		\item[(iii')] for every $r\geq1$, $n\in \N$, $\xastast_1,\ldots ,\xastast_n\in rB_{\Xastast}$, and $\varepsilon>0$ there is $y\in B_{X}\setminus\{0\}$ such that $$\Vert \xastast_i +y\Vert \geq (1-\varepsilon)(\Vert \xastast_i\Vert +1)\ \quad \text{for every $i\in\{1,\dots,n \}.$}$$
	\end{itemize}
\end{lemma}
\begin{proof}
	The equivalence between (i) and (ii) is done in \cite{haller}. Obviously (ii') $\Rightarrow$ (ii), (iii') $\Rightarrow$ (iii), and (i) $\Rightarrow$ (iii). 
	 
 (iii')$\Rightarrow$(ii'). Let $n\in \N$, $\xastast_1,\dots, \xastast_n\in S_{\Xastast}$, and $\varepsilon>0$. By (iii') there is a $y\in B_X\setminus\{0\}$ such that 
 \[
 \|\xastast_i +y\|>(1-\frac{\varepsilon}4)(1 +1)=2-\frac{\varepsilon}2 \quad \text{for every $i\in\{1,\dots, n\}$}.
 \]
 Therefore $\Vert y\Vert>1-\frac{\varepsilon}2$ and
 \[
 \|\xastast_i +\frac{y}{\Vert y\Vert}\|\geq \|\xastast_i +y\|-\Vert y-\frac{y}{\|y\|}\Vert\geq 2-\frac{\varepsilon}2-(1-\|y\|)>2-\varepsilon
 \]
 for every $i\in\{1,\dots, n\}$.
  
 (iii)$\Rightarrow$(iii'). Let $r\geq1$, $n\in \N$, $\xastast_1,\dots, \xastast_n\in rB_{\Xastast}$, and $\varepsilon>0$. By (iii) there is a $\yastast\in B_{\Xastast}\setminus\{0\}$ such that 
 \[
 \|\xastast_i +\yastast\|>(1-\varepsilon)(\|\xastast_i\| +1) \quad \text{for every $i\in\{1,\dots, n\}$}.
 \]
 By Goldstine's theorem there is a net $\{y_\lambda\}\subset B_X\setminus\{0\}$ such that $y_\lambda$ weak$^*$ converges to $\yastast$. Finally, by the weak$^*$ lower semicontinuity of the norm in $\Xastast$, we deduce that there is a $\lambda_0$ such that 
  \[
 \|\xastast_i +y_{\lambda_0}\|>(1-\varepsilon)(\|\xastast_i\| +1) \quad \text{for every $i\in\{1,\dots, n\}$}.
 \]
	\end{proof}

 The equivalence between (i) and (iii') in Lemma~\ref{lemma: bidual octahedral} motivates the following definition.

\begin{definition}
	Let $X$ be a Banach space and fix $B$ a closed, convex, and bounded subset of $X^{**}$. A subset $A\subset B_X\setminus\{0\}$ is called an \emph{octahedral set for $B$} if for every $n\in \N$, $\xastast_1,\dots, \xastast_n\in B$, and $\varepsilon>0$ there is an element $a\in A$ such that 
	\[
	\Vert \xastast_i + a \Vert\geq (1-\varepsilon)(\Vert \xastast_i\Vert +1)\quad \text{for every $i\in\{1,\dots,n \}.$}
	\]
	In the case $B=B_{X^{**}}$ we will say that $A$ is an octahedral set for $X^{**}$, without mentioning $B_{X^{**}}$.
\end{definition}
Observe that an octahedral set for $X^{**}$ is a subset of $X$ giving the octahedrality in $X^{**}$ and so in $X$.

As we say in the introduction there is a complete duality relation between octahedrality and strong diameter two property. In fact, a Banach space $X$ satisfies the strong diameter two property (SD2P in short) if, and only if, $X^*$ is octahedral. The next lemma uses this dual relation to get octahedral subsets from SD2P subsets.
\begin{lemma}\label{lemma: SD2P implies OH in the dual}
	Let $X$ be a Banach space. Assume that there is $A\subset B_X\setminus\{0\}$  such that for every $n\in \N$, $0<\varepsilon<2$, and every average of slices of $B_{X^*}$, $\frac1n \sum_{i=1}^{n} S(B_{\Xast}, \xastast_i, \varepsilon)$,  there exist $\xast_i, \yast_i\in S(B_{\Xast}, \xastast_i, \varepsilon)$ and $x\in A$ such that
	$$\left(\frac1n\sum_{i=1}^{n}\xast_i-\frac1n\sum_{i=1}^{n}\yast_i\right)(x)>2-\varepsilon.$$
	
	Then for every $n\in \N$, $\xastast_1,\dots, \xastast_n\in S_{\Xastast}$, and $\varepsilon>0$ there exists $y\in A$ such that 
	\[
	\|\xastast_i +y\|>2-\varepsilon \quad \text{for every $i\in\{1,\dots, n\}$},
	\] and so $A$ is an octahedral set for $X^{**}$.
\end{lemma}
\begin{proof}
	Let $n\in \N$, $\xastast_1,\dots, \xastast_n\in S_{\Xastast}$, and $\varepsilon>0$. Consider now the convex combination of slices $\frac1n \sum_{i=1}^{n} S(B_{\Xast}, \xastast_i, \frac{\varepsilon}{2n})$, then by our assumption there are $\xast_i, \yast_i\in S(B_{\Xast}, \xastast_i, \frac{\varepsilon}{2n})$ and a $x\in A$ such that 
	$$\left(\frac1n\sum_{i=1}^{n}\xast_i-\frac1n\sum_{i=1}^{n}\yast_i\right)(x)>2-\frac{\varepsilon}{2n}.$$
	Then $x^*_i(x)-y^*_i(x)>2-\frac{\varepsilon}{2}$ for every for every $i\in\{1,\dots,n \}$ and so $x^*_i(x)>1-\frac{\varepsilon}{2}$ for every for every $i\in\{1,\dots,n \}.$
	
	Finally, we obtain that 
	$$\Vert x_i^{**}+x\Vert\geq x_i^{**}(x_i^*)+x_i^*(x)>1-\frac{\varepsilon}{2n}+1-\frac{\varepsilon}{2}\geq 2-\varepsilon.$$
	Therefore, by the equivalence between (ii') and (iii') in Lemma~\ref{lemma: bidual octahedral}, $A$ is an octahedral set for $\Xastast$.
\end{proof}
Our strategy will be to renorm a Banach space with a bidual octahedral norm, starting from another Banach space with a good octahedral subset. The next proposition is a stability property under renorming, in this direction.

\begin{proposition}\label{prop: abstract bidual renorming}
	Let $(X, \Vert \cdot \Vert_X)$ be a Banach space. Assume that there exists a Banach space $(Y, \Vert \cdot \Vert_Y)$ and a bounded linear operator $S\colon X \to Y$ such that the following conditions hold:
	\begin{itemize}
		\item[(A1)] there exists an octahedral set $B\subset B_Y$ for $\Yastast$;
		\item[(A2)] there exists $Z$ a subspace of $X$ such that $\left.S\right|_{Z}\colon (Z, \|\cdot\|_X)\to (\widehat{B}, \|\cdot\|_Y)$ is an onto isomorphism, where $\widehat{B}$ denote the closed linear span $B$.
	\end{itemize}
	
	Then there is an equivalent norm in $X$ such that the bidual $X^{\ast\ast}$ is octahedral.
\end{proposition}
\begin{proof}
	As the restriction of $S$ to $Z$ is an onto isomorphism, we can renorm $Z$ so that $S$ is a onto isometry. Now, we can extend the above norm to an equivalent norm $\Vert \cdot \Vert$ on $X$.
	Define a new norm on $X$ by 
	$$
	p(x)=\Vert Sx\Vert_Y +\Vert x+Z\Vert \quad \text{for all $x\in X$}.
	$$ 
	
	It is clear that $p$ is a seminorm. As $\left.S\right|_{Z}$ is  one to one, then $p$ is a norm. Denote by $r:= \Vert S\Vert$. In order to see that $p$ and $\Vert\cdot\Vert$ are equivalent norms, we show that $\frac1{r+2} \Vert x\Vert \leq p(x)\leq (r+1)\Vert x\Vert$ for every $x\in X$. The upper estimate is clear. For the lower estimate, assume that $x\in X$ with $\Vert x\Vert=1$. If $\Vert x+Z\Vert \geq \frac1{r+2}$ then it is clear that $p(x)\geq \frac{\Vert x\Vert}{r+2}$. Assume now that $\Vert x+Z\Vert < \frac1{r+2}$, then there is a $z\in Z$ with $\Vert x-z\Vert <\frac1{r+2}$. Hence, $\Vert z \Vert\geq \Vert x \Vert -\Vert x-z\Vert>\frac{r+1}{r+2}$ and
	\begin{align*}
	p(x)&\geq \Vert Sx\Vert_Y\geq \Vert Sz\Vert_Y-\Vert S(x-z)\Vert_Y\\
	&\geq \Vert z \Vert -r\Vert x-z\Vert\geq \frac{r+1}{r+2}-\frac{r}{r+2}=\frac{\Vert x\Vert}{r+2}.
	\end{align*}
	Note that $(X,p)$ is isometric to a subspace of $Y\oplus_1 X/Z$, thus its bidual $(X, p)^{\ast\ast}$ is isometric to a subspace of $\Yastast\oplus_1 X^{**}/Z^{\perp\perp}$. Denote by $\tilde{p}$ the bidual norm of $(X, p)^{\ast\ast}$, that is, 
	\[
	\tilde{p}(x^{\ast\ast})=\Vert S^{\ast\ast}x^{\ast\ast}\Vert_{\Yastast}+\Vert x^{\ast\ast}+Z^{\perp\perp}\Vert \quad \text{for all $x^{\ast\ast}\in X^{\ast\ast}$}. 
	\]
	
	Finally, we prove that $(X^{**},\tilde{p})$ has an octahedral norm. Let $n\in \N$, $x^{\ast\ast}_1, \dots, x^{\ast\ast}_n$ in $X^{\ast\ast}$ with $\tilde{p}(x^{\ast\ast}_i)=1$, and $\varepsilon>0$. By assumption (A1), $B$ is an octahedral set for $\Yastast$, thus we can find $b\in B$ such that $\Vert S^{\ast\ast}x^{\ast\ast}_i+b\Vert_{\Yastast} \geq \Vert S^{\ast\ast}x^{\ast\ast}_i\Vert_{\Yastast}+1-\varepsilon$ for every $i\in\{1,\dots,n\}$. Find now an element $z\in Z$ such that $b=Sz$. Since $S^{\ast\ast}z=Sz=b$, we have now
	\begin{align*}
	\tilde{p}(x^{\ast\ast}_i+z)&=\Vert S^{\ast\ast}x^{\ast\ast}_i+S^{\ast\ast}z\Vert_{\Yastast} + \Vert x^{\ast\ast}_i+z+Z^{\perp\perp}\Vert\\
	&\geq \Vert S^{\ast\ast}x^{\ast\ast}_i\Vert_{\Yastast}+1-\varepsilon+\Vert x^{\ast\ast}_i+Z^{\perp\perp}\Vert\\
	&= \tilde{p}(x^{\ast\ast}_i)+1-\varepsilon=2-\varepsilon
	\end{align*}
	for every $i\in\{1,\dots,n\}$.
\end{proof}

 The next proposition gives a dual Banach space $Y$ from a weak$^{\ast}$ compact and convex set $C$ inside another dual Banach space $X^*$ such that the unit ball of $Y$ is $\co(C\cup-C)$. This result is essentially known, but we give the proof here for sake of completeness.

\begin{proposition}\label{prop: renorm subspace SD2P}
	Let $X$ be a Banach space, $C\subset B_{\Xast}$ a weak$^{\ast}$ compact and convex set, and $Y$ be the linear span of $C$. Then there is a (non necessarily equivalent) dual norm $|\cdot|$ in $Y$ satisfying $\vert\cdot\vert\geq \Vert\cdot\Vert_Y$ such that $B_{(Y,\vert\cdot\vert)}=\co(C\cup-C)$. Furthermore, if $i$ is the inclusion map from $(Y,|\cdot|)$ into $X^*$, and $i^*$ is the adjoint operator from $X^{**}$ into $(Y,|\cdot|)^*$, then $i^*(X)$ is a dense subspace of $(Y,|\cdot|)_*$, the predual of $(Y,|\cdot|)$, with $\Vert i^*\Vert\leq 1$. Finally, if $Y$ is closed in $X^*$, then $Y$ is $w^*$-closed and $|\cdot|$ and $\Vert\cdot\Vert_Y$ are equivalent norms.  
\end{proposition}
\begin{proof}
	Define $K:=\co(C\cup -C)$. Now, for every $\varepsilon>0$, we define an equivalent norm $\|\cdot\|_\varepsilon$ in $\Xast$ whose new unit ball is the set $B_{(\Xast,\|\cdot\|_\varepsilon)}=K+\varepsilon B_{\Xast}$.
	
	Denote by
	$$
	Z:=\{\xast\in \Xast\colon \sup_{\varepsilon>0}\|\xast\|_\varepsilon <\infty \}
	$$
	and define a norm on $Z$ by $|\xast|:=\sup_{\varepsilon>0}\|\xast\|_\varepsilon$ for every $\xast \in Z$. Observe that $B_{(Z,|\cdot|)}\subset (1+\varepsilon)B_Z$ for every $\varepsilon$ and thus $|\cdot|\geq \|\cdot\|_Z$. Note that $(Z,\vert\cdot\vert)$ is isometric to the diagonal subspace of the $\ell_{\infty}$-sum of the family of Banach spaces $\{(X^*,\Vert\cdot\Vert_{\varepsilon}):\varepsilon>0\}$. Then $Z$ is a Banach space.
	
	Now $B_{(Z,|\cdot|)}=K$, that is,  $K= \bigcap_{\varepsilon>0} (K+\varepsilon B_{\Xast})$. Indeed, clearly $K\subset K+\varepsilon B_{\Xast}$ for every $\varepsilon>0$. For the converse, if $\xast \in \bigcap_{\varepsilon>0} (K+\varepsilon B_{\Xast})$, then there are $k_\varepsilon\in K$ and $\xast_\varepsilon\in B_{\Xast}$ such that $\xast=k_\varepsilon+\varepsilon \xast_\varepsilon$ for every $\varepsilon>0$. Hence, $\|\xast-k_\varepsilon\|= \varepsilon\|\xast_\varepsilon\|\leq \varepsilon$ for every $\varepsilon>0$. Thus, $d(\xast, K)=0$ and $\xast\in K$, because $K$ is closed.  
	
	From the equality $B_{(Z,|\cdot|)}=K$, we get that $Z=\lin B_{(Z,|\cdot|)}=\lin K=Y$. Then $(Y,|\cdot|)$ is a Banach space whose unit ball is compact in a locally convex and separated topology in $Y$, the weak-star topology of $X^*$ on $Y$, and so $(Y,|\cdot|)$ is a dual space whose predual is the closure in $(Y,|\cdot|)^*$ of $X_{\mid Y}:=\{x_{\mid Y}:x\in X\subset X^{**}\}$ (see \cite{kaijser}). Now  it is clear that $i^*(X)$ is a dense subspace of $(Y,|\cdot|)_*$ and $\Vert i^*\Vert\leq 1$, since $\Vert i\Vert \leq 1$. In the case  $Y$ is closed in $X^*$, we have that $|\cdot|$ is an equivalent norm in $Y=Z$, since $Y$ is closed in $X$, $(Z,\vert\cdot\vert)$ is complete and $|\cdot|\geq \|\cdot\|_Z$, and so $Y$ is $w^*$-closed, applying Banach-Dieudonn\'e Theorem. 
\end{proof}

\section{Talagrand set}

We start introducing some notations and results from \cite[Theorem~4.6]{schachermayer_moduli_1989}, where a "good" subset in $C(\Delta)^*$ is constructed with the SD2P. This construction, completed by M. Talagrand, will be crucial to get our main results.

Consider a natural number $s\geq 3$ and let $(N_s)_{s\geq 3}$ be a partition of $\N$ into disjoint infinite sets. Now fix $s\geq 3$. For $I\subset \N$, $i\in \N$, define on the $i$th copy of $\{0,1\}$, denoted by $\{0,1\}_i$, a measure
$$
\nu^{(i)}_{s,I} = \begin{cases} \frac1s\delta^{(i)}_0+\frac{s-1}s\delta^{(i)}_1 &\mbox{if } i \in I \\ 
\frac{s-1}s\delta^{(i)}_0+\frac1s\delta^{(i)}_1 & \mbox{if } i \notin I. \end{cases} 
$$

Now for $J\subset \N$, $I\subset J$, and $p\in J$, define
$$
\mu^{J}_{s,I}=\underset{i\in J}{\bigotimes}\nu_{s,I}^{(i)},
$$
$$
\bar{\mu}^{J,p}_{s,I}=\delta^{(p)}_0\otimes\left(\underset{i\in J\setminus\{p\}}{\bigotimes}\nu_{s,I}^{(i)}\right), \quad
\bar{\bar{\mu}}^{J,p}_{s,I}=\delta^{(p)}_1\otimes\left(\underset{i\in J\setminus\{p\}}{\bigotimes}\nu_{s,I}^{(i)}\right),
$$
and
$$
\rho^J_{I}=\underset{s\geq 3}{\bigotimes}\mu^{J\cap N_s}_{s,I\cap N_s}.
$$
If $J=\N$, then we will use the notation $\rho_I:=\rho^{\N}_{I}$.

We also consider the operator $T_J\colon C(\Delta_J)\to C(\Delta_J)$ defined by $T_J(f)(I)=\rho^J_{I}(f)$, for every $I\subset J$, where we have identified $\Delta_J$ with the power set $\mathcal{P}(J)$. One can easily check that $T^*_J\colon \mathcal{M}(\Delta_J)\to \mathcal{M}(\Delta_J)$ satisfies $T^*_J(\delta^J_I)=\rho^J_{I}$, where $\delta^J_I$ is the Dirac measure at $I$ on $\Delta_J$, and then $T^{\ast}_J(\theta)=w^{\ast}\mbox{-}\int_{\Delta_J}\rho^J_{I}\;d\theta(I)$.
If $J=\N$, then we will use the notation $T:=T_{\N}$.


Note that the identification between $\Delta$ and $\mathcal{P}(\N)$ is done in the following way: if $x\in \{0,1\}^{\N}$ then we see $x$ as the element in $\Delta$ given by $I_x=\{n\in \N:x(n)=1\}$; if $I\subset \N$ then we can see $I$ as the element $x_I\in \Delta$ given by $x_I(n)=1$ if $n\in I$ and $x_I(n)=0$  otherwise. In this way, we have
$$
\nu^{(i)}_{s,I\cap N_s } = \begin{cases} \frac1s\delta^{(i)}_0+\frac{s-1}s\delta^{(i)}_1 &\mbox{if }\  x_I(i)=x_{N_s}(i)=1 \\ 
\frac{s-1}s\delta^{(i)}_0+\frac1s\delta^{(i)}_1 &  {\rm otherwise}. \end{cases} 
$$

Denote by $P_{\Delta}$ the set of probability measures on $\Delta$. In \cite[Theorem~4.6]{schachermayer_moduli_1989} it was shown that $\mathcal{C}:=T^{\ast}(P_{\Delta})$ is a convex $w^{\ast}$-compact set of $P_{\Delta}$ with the property that every convex combination of weak slices of $\mathcal{C}$ has diameter two. 

The next lemma and proposition are part of the proof of {\cite[Theorem~4.6]{schachermayer_moduli_1989}} and give one of the tools to find a key octahedral set. We include it here for sake of completeness. 
\begin{lemma}[see {\cite[Lemma~4.8]{schachermayer_moduli_1989}}]\label{lemma: existence of p}
	For every $s\geq 3$, every $\delta>0$, there exists $k=k(s,\delta)$ such that: 
	
	For every $n\in \N$, every $J\subset \N$, $\vert J\vert\geq kn$, every $(\varphi_i)_{i\in \{1,\dots,n\}}\in \mathcal{M}(\Delta_J)^*$, $\Vert \varphi_i\Vert \leq 1$, there exists $p\in J$ such that 
	\[
	\sup_{I\in \{\emptyset,J\}}\sup_{i\in\{1,\dots,n\}} \vert\varphi_i(\bar{\mu}^{J,p}_{s,I}-\bar{\bar{\mu}}^{J,p}_{s,I})\vert<\delta.
	\]
\end{lemma}

\begin{proposition}[see the proof of {\cite[Theorem~4.6]{schachermayer_moduli_1989}}]\label{prop: C has SD2P}
Let $n\in \N$ and $S_1,\dots, S_n$ be slices of $\mathcal{C}$. Then for every $s\geq 3$ there exist $J\subset \N$, $|J|\geq n$, and $p\in J$ such that 
\[
\sigma_1:=\frac{1}{n}\left(\sum_{i=1}^{l}\mu^J_{s,J}\otimes \bar{\gamma}_i+\sum_{i=l+1}^{n}\mu^J_{s,\{p\}}\otimes \bar{\gamma}_i\right)\in \frac{1}{n}\sum_{i=1}^{n}S_i,
\]
\[
\sigma_2:=\frac{1}{n}\left(\sum_{i=1}^{l}\mu^J_{s,J\setminus\{p\}}\otimes \bar{\gamma}_i+\sum_{i=l+1}^{n}\mu^J_{s,\{\emptyset\}}\otimes \bar{\gamma}_i\right)\in \frac{1}{n}\sum_{i=1}^{n}S_i,
\]
where $l\in \{1,\dots, n\}$ and $\bar{\gamma}_i$ is a probability measure for every $i\in \{1,\dots,n\}$. 
\end{proposition}
\begin{proof}

	Observe first that for every $J\subset \N$, $\vert J\vert <\infty$, every $\sigma\in \mathcal{C}$ is a convex combination of elements of $\mathcal{C}$ of the form $\rho^J_I\otimes \gamma_I$, where $I$ runs through the subsets of $J$, and $\gamma_I=T^*_{\N\setminus J}(\theta_I)$ for some probability measure $\theta_I$ on $\Delta_{\N\setminus J}$.
	
	Let $n\in \N$, $\delta>0$, $\varphi_1,\dots,\varphi_n\in \mathcal{M}(\Delta)^*$ with $\Vert \varphi_i\Vert=1$, and consider
	\[
	S(\mathcal{C}, \varphi_i, \delta)=\{\sigma\in \mathcal{C}\colon \varphi_i(\sigma)>M_i-\delta \},
	\]
	where $M_i=\sup_{\sigma\in \mathcal{C}}\varphi_i(\sigma)$.
	
	For every $s\geq 3$, let $J_0\subset N_s$, $\vert J_0\vert =n2^nk$, where $k=k(s,\delta/2)$ is given by Lemma \ref{lemma: existence of p}. For every $i\in\{1,\dots,n \}$, choose $\sigma_i\in \mathcal{C}$ such that $\varphi_i(\sigma_i)>M_i-\delta/2$. By the above observation (applied to $\sigma_i$) and a convexity argument, we may find $I_i\subset J_0$ such that $\varphi_i(\rho^{J_0}_{I_i}\otimes \gamma_{i, I_i})>M_i-\delta/2$.
	
	By a cardinality argument, there exists $J\subset J_0$, $\vert J\vert\geq nk(s,\delta/2)$, which satisfies either $J\subset I_i$ or $J\subset J_0\setminus I_i$ for every $i\in \{1,\dots,n\}$. If we put $\bar{\gamma}_i=\mu^{J_0\setminus J}_{s,I_i\setminus J}\otimes \gamma_{i, I_i}$, we have, since $J_0\subset N_s$, that $\rho^{J_0}_{I_i}\otimes \gamma_{i, I_i}=\mu^J_{s,J\cap I_i}\otimes \bar{\gamma}_i.$
	
	Define now elements $\psi_i\in \mathcal{M}(\Delta_J)^*$ by $\psi_i(\sigma)=\varphi_i(\sigma\otimes \bar{\gamma}_i)$, and apply Lemma \ref{lemma: existence of p} to find a $p\in J$ such that 
	\[
	\sup_{I\in \{\emptyset,J\}}\sup_{i\in\{1,\dots,n\}} \psi_i(\bar{\mu}^{J,p}_{s,I}-\bar{\bar{\mu}}^{J,p}_{s,I})<\delta/2.
	\]
	
	Without loss of generality we can suppose that $J\subset I_i$ for $i\in \{1,\dots,l\}$ and $J\subset J_0\setminus I_i$ for $i\in \{l+1,\dots,n\}$. By a convexity argument we deduce that
	\[
	\{\mu^J_{s,J}\otimes \bar{\gamma}_i, \mu^J_{s,J\setminus\{p\}}\otimes \bar{\gamma}_i\}\subset S_i \quad \text{for $i\in \{1,\dots,l\}$}
	\]
	and
	\[
	\{\mu^J_{s,\{p\}}\otimes \bar{\gamma}_i, \mu^J_{s,\emptyset}\otimes \bar{\gamma}_i\}\subset S_i \quad \text{for $i\in \{l+1,\dots,n\}$}.
	\]
	 \end{proof}

Let $p\in \N$ and, in the following, denote by $f_p\colon \Delta\to \R$ the function
$$
f_p(x) = \begin{cases} -1 &\mbox{if } x(p)=0 \\ 
1 & \mbox{if } x(p)=1, \end{cases} 
$$
where $x\in \Delta$. Note that $f_p\in C(\Delta)$ with $\Vert f_p\Vert=1$ for every $p\in \N$. 

In the following, we denote by $\mathcal{Y}$  the linear span of $\mathcal{C}$ and $\vert\cdot\vert$ the norm on $\mathcal{Y}$ given by Proposition \ref{prop: renorm subspace SD2P}. Recall that $\mathcal{Y}$ is a dual Banach space satisfying $B_{(\mathcal{Y},\vert\cdot\vert)}=\mathcal{K}={\rm co}(\mathcal{C}\cup -\mathcal{C})$. Denote by $(\mathcal{Y}_*, \vert\cdot \vert_*)$, the predual of $(\mathcal{Y}, \vert\cdot \vert)$. 

\begin{lemma}\label{lemma: f_p is a good set}
	The set $\{{f_p}_{\mid \mathcal{Y}}\colon p\in \N \}$ is an octahedral set for $(\mathcal{Y}_*, \vert\cdot \vert_*)^{**}$. In particular, $(\mathcal{Y}, \vert\cdot \vert)$ has the strong diameter two property.
\end{lemma}
\begin{proof}
	Note first that $B_\mathcal{Y}=\mathcal{K}=\{\lambda a-(1-\lambda)b\colon \lambda\in [0,1], a,b\in \mathcal{C} \}$. By Lemma \ref{lemma: SD2P implies OH in the dual}, it is enough to prove that $\mathcal{C}$ satisfies the following property: 
	
	($\bullet$) for every $n\in \N$, $\varepsilon>0$, and every average of slices of $\mathcal{C}$, $\frac1n \sum_{i=1}^{n} S(\mathcal{C}, \xastast_i, \varepsilon)$, there exist $\xast_i, \yast_i\in S(\mathcal{C}, \xastast_i, \varepsilon)$ and $p\in \N$ such that
	\[
	\frac{1}{n}\sum_{i=1}^n(x_i^*-y_i^*)({f_p}_{\mid \mathcal{Y}}	
	)>2-\varepsilon.
	\]
	
	Indeed, assume that $\{{f_p}_{\mid \mathcal{Y}}\}_{p\in \N}$ satisfies ($\bullet$) and let us show that  $\{{f_p}_{\mid \mathcal{Y}}\colon p\in \N \}$ is an octahedral set for $(\mathcal{Y}_*, \vert\cdot \vert_*)^{**}$. Fix $n\in \N$, $\varepsilon>0$, and $S_1,\ldots ,S_n$ slices in $\mathcal{K}$. Observe that every slice $S_i$ has nonempty intersection with $\mathcal{C}$ or $-\mathcal{C}$, since $\mathcal{K}={\rm co}(\mathcal{C}\cup -\mathcal{C})$. Let $I:=\{i\in \{1,\ldots n\}:S_i\cap \mathcal{C}\neq \emptyset\}$ and $J:= \{1,\ldots n\}\setminus I$. We consider
	$$A:=\frac{1}{n}\sum_{i\in I}S_i\cap \mathcal{C}+\frac{1}{n}\sum_{i\in J}S_i\cap(-\mathcal{C})\subset \frac{1}{n}\sum_{i=1}^n S_i.$$
	Observe that 
	\begin{align*}
	A-A&=\frac{1}{n}\sum_{i\in I}S_i\cap \mathcal{C}+\frac{1}{n}\sum_{i\in J}S_i\cap(-\mathcal{C})-\frac{1}{n}\sum_{i\in I}S_i\cap \mathcal{C}-\frac{1}{n}\sum_{i\in J}S_i\cap(-\mathcal{C})\\
	&=\frac{1}{n}\sum_{i\in I}S_i\cap \mathcal{C}+\frac{1}{n}\sum_{i\in J}(-S_i)\cap\mathcal{C}-\left(\frac{1}{n}\sum_{i\in I}S_i\cap \mathcal{C}+\frac{1}{n}\sum_{i\in J}(-S_i)\cap\mathcal{C}\right)\\
	&=B-B,
	\end{align*}
	where $B:= \frac{1}{n}\sum_{i\in I}S_i\cap \mathcal{C}+\frac{1}{n}\sum_{i\in J}(-S_i)\cap\mathcal{C}$ is an average of slices in $\mathcal{C}$. Since $A-A=B-B$ and as we are assuming  ($\bullet$), we have that
	\[
	\sup_{p\in \N}\sup_{\zast\in A-A}\zast({f_p}_{\mid \mathcal{Y}})=\sup_{p\in \N}\sup_{\zast\in B-B}\zast({f_p}_{\mid \mathcal{Y}})\geq 2-\varepsilon/2.
	\]
	Therefore there are $x_i^* ,y_i^*\in S_i\cap \mathcal{C}\subset S_i$ for every $i\in I$,$x_i^* ,y_i^*\in S_i\cap (-\mathcal{C})\subset S_i$ for every $i\in J$, and $p\in \N$ such that
	$$\frac{1}{n}\sum_{i=1}^n(x_i^*-y_i^*)(f_p)>2-\varepsilon,$$
	which, by Lemma \ref{lemma: SD2P implies OH in the dual}, proves that $\{{f_p}_{\mid \mathcal{Y}}\colon p\in \N \}$ is an octahedral set  for $(\mathcal{Y}_*, \vert\cdot \vert_*)^{**}$
	
Let's prove now that $\{{f_p}_{\mid \mathcal{Y}}\}_{p\in \N}$ is  an octahedral set  for $(\mathcal{Y}_*, \vert\cdot \vert_*)^{**}$. For this we will apply again Lemma \ref{lemma: SD2P implies OH in the dual}.

	Let $n\in \N$, $S_1,\dots, S_n$ be slices of $\mathcal{C}$, and $\varepsilon>0$. Find $s\geq 3$ such that $4/s< \varepsilon$. By Proposition~\ref{prop: C has SD2P} there exist $J\subset \N$, $|J|\geq n$, and $p\in J$ such that 
	\[
	\sigma_1:=\frac{1}{n}\left(\sum_{i=1}^{l}\mu^J_{s,J}\otimes \bar{\gamma}_i+\sum_{i=l+1}^{n}\mu^J_{s,\{p\}}\otimes \bar{\gamma}_i\right)\in \frac{1}{n}\sum_{i=1}^{n}S_i,
	\]
	\[
	\sigma_2:=\frac{1}{n}\left(\sum_{i=1}^{l}\mu^J_{s,J\setminus\{p\}}\otimes \bar{\gamma}_i+\sum_{i=l+1}^{n}\mu^J_{s,\{\emptyset\}}\otimes \bar{\gamma}_i\right)\in \frac{1}{n}\sum_{i=1}^{n}S_i,
	\]
	where $l\in \{1,\dots, n\}$ and $\bar{\gamma}_i$ is a probability measure for every $i\in \{1,\dots,n\}$. 
	
		Therefore
	$$ \langle\sigma_1-\sigma_2, f_p\rangle$$ $$=
	\Bigg\langle\frac{1}{n}\left(\sum_{i=1}^{l}\mu^J_{s,J}\otimes \bar{\gamma}_i+\sum_{i=l+1}^{n}\mu^J_{s,\{p\}}\otimes \bar{\gamma}_i\right)$$
	$$-\frac{1}{n}\left(\sum_{i=1}^{l}\mu^J_{s,J\setminus\{p\}}\otimes \bar{\gamma}_i+\sum_{i=l+1}^{n}\mu^J_{s,\{\emptyset\}}\otimes \bar{\gamma}_i\right),f_p\Bigg\rangle$$
	$$=\frac{1}{n}\Bigg\langle(\mu^J_{s,J}-\mu^J_{s,J\setminus\{p\}})\otimes \left(\sum_{i=1}^{l}\bar{\gamma}_i\right)+(\mu^J_{s,\{p\}}-\mu^J_{s,\emptyset})\otimes \left(\sum_{i=l+1}^{n}\bar{\gamma}_i\right) ,f_p\Bigg\rangle$$
	$$= \frac{s-2}{s}\Bigg\langle(\delta^{(p)}_1-\delta^{(p)}_0)\otimes \Bigg[\frac{l}{n}\left(\mu^{J\setminus\{p\}}_{s,J\setminus\{p\}}\otimes\left(\frac{1}{l}\sum_{i=1}^{l}\bar{\gamma}_i\right) \right)$$
	$$+\frac{n-l}{n}\left(\mu^{J\setminus\{p\}}_{s,\emptyset}\otimes\left(\frac{1}{n-l}\sum_{i=l+1}^{n}\bar{\gamma}_i\right) \right)\Bigg] ,f_p\Bigg\rangle$$
	$$=2\frac{s-2}{s}> 2-\varepsilon.$$
	In the last equality we use the Fubini theorem, and the fact that the measure which appears between brackets
	$$\Bigg[\frac{l}{n}\left(\mu^{J\setminus\{p\}}_{s,J\setminus\{p\}}\otimes\left(\frac{1}{l}\sum_{i=1}^{l}\bar{\gamma}_i\right) \right)
	+\frac{n-l}{n}\left(\mu^{J\setminus\{p\}}_{s,\emptyset}\otimes\left(\frac{1}{n-l}\sum_{i=l+1}^{n}\bar{\gamma}_i\right) \right)\Bigg] $$ is a probability measure, since it is a convex combination of probability measures.	
	
	From the dual relation of an octahedral norm and the strong diameter two property \cite{becerra_octahedral_2014}, we get that $(\mathcal{Y}, \vert\cdot \vert)$ is a dual Banach space with the strong diameter two property.
\end{proof}

Denote by $j$ the canoncial embedding of $C(\Delta)$ into $C(\Delta)^{**}$. Observe now that, from Proposition \ref{prop: renorm subspace SD2P} we  have for every $f\in C(\Delta)$ that

$$f_{\mid \mathcal{Y}}\in \mathcal{Y}_*\ \text{and}\ \vert f_{\mid \mathcal{Y}}\vert_*\leq \Vert f_{\mid \mathcal{Y}}\Vert_{(\mathcal{Y},\Vert\cdot\Vert_{C(\Delta)^*})^{*}}.$$

In the following lemma we will collect some of the properties of $f_p$. In particular, the linear spans of $f_p$ in $C(\Delta)$, $\left.j(C(\Delta))\right|_\mathcal{Y}$, and $\mathcal{Y}_*$ are all isomorphic to $\ell_1$.  
 
\begin{lemma}\label{lemma: isormphic copy of ell_1}
	The functions $f_p$ satisfy the following:
	\begin{enumerate}
		\item $1=\|f_p\|_{C(\Delta)}\geq \|\left.f_p\right|_\mathcal{Y}\|_{(\mathcal{Y},\Vert\cdot\Vert_{C(\Delta)^*})^*}\geq |{f_p}_{\mid \mathcal{Y}}|_*$ for every $p\in \N$.
		\item For every $n\in \N, \alpha_1,\dots, \alpha_n\in \R$ 
		$$
	\Big	\|\sum_{i=1}^{n}\alpha_{i} f_{p_i}\Big\|_{C(\Delta)}=\sum_{i=1}^{n}|\alpha_{i}|.
		$$
		\item For every $n\in \N, \alpha_1,\dots, \alpha_n\in \R$
		$$
		\sum_{i=1}^{n}|\alpha_{i}|\geq\Big\|\sum_{i=1}^{n}\alpha_{i} \left.f_{p_i}\right|_\mathcal{Y}\Big\|_{(\mathcal{Y},\Vert\cdot\Vert_{C(\Delta)^*})^*}\geq \Big|\sum_{i=1}^{n}\alpha_{i} {f_{p_i}}_{\mid \mathcal{Y}}\Big|_*\geq \frac13\sum_{i=1}^{n}|\alpha_{i}|.
		$$
	\end{enumerate}	
\end{lemma}
\begin{proof}
(1) is clear from the preceeding comment.

(2). Let $n\in \N$, $\alpha_1,\dots, \alpha_n\in \R\setminus\{0\}$, and $p_1,\dots, p_n\in \N$. Since $\|f_{p_i}\|_{C(\Delta)}=1$, we clearly have that $\|\sum_{i=1}^{n}\alpha_{i} f_{p_i}\|_{C(\Delta)}\leq \sum_{i=1}^{n}|\alpha_{i}|$. For the other inequality, choose an $x_0\in \Delta$ such that 
\[x_0(p_i)=\begin{cases} 0 &\mbox{if }\  \alpha_{i}<0 \\ 
1 & \mbox{if } \alpha_i>0. \end{cases} \]
Hence,
\[
\Big\|\sum_{i=1}^{n}\alpha_{i} f_{p_i}\Big\|_{C(\Delta)}\geq \sum_{i=1}^{n}\alpha_{i} f_{p_i}(x_0)=\sum_{i=1}^{n}|\alpha_{i}|.
\]
(3). Let $n\in \N$, $\alpha_1,\dots, \alpha_n\in \R\setminus\{0\}$, and $p_1,\dots, p_n\in \N$. Find $s_1,\dots, s_n\geq 3$ such that $p_i\in N_{s_i}$. Observe that the inequalities 	
$$
\sum_{i=1}^{n}|\alpha_{i}|\geq\Big\|\sum_{i=1}^{n}\alpha_{i} \left.f_{p_i}\right|_\mathcal{Y}\Big\|_{(\mathcal{Y},\Vert\cdot\Vert_{C(\Delta)^*})^*}\geq \Big|\sum_{i=1}^{n}\alpha_{i} {f_{p_i}}_{\mid \mathcal{Y}}\Big|_*
$$
are clear, because $\|f_{p_i}\|_{C(\Delta)}=1$ and $\|\cdot\|_{(\mathcal{Y},\Vert\cdot\Vert_{C(\Delta)^*})^*}\geq |\cdot|_*$ on $j(C(\Delta)_{\mid \mathcal{Y}}$. For the lower estimate choose  $I\subset \N$ such that $p_i\in I$ if $\alpha_i>0$ and $p_i\notin I$ if $\alpha_{i}<0$. Consider the measure
$$
\rho_{I}=\underset{s\geq 3}{\bigotimes}\underset{i\in N_s}{\bigotimes}\nu_{s,I\cap N_s }^{(i)}
$$
and denote by $$
\mu_k:=\underset{s\geq 3}{\bigotimes}\underset{i\in N_s\setminus \{p_k\}}{\bigotimes}\nu_{s,I\cap N_s }^{(i)}.
$$
Note that $\mu_k$ is a probability measure for every $k\in\{1,\dots,n\}$.
By the Fubini theorem we have 
\begin{align*}
\Bigg\vert \sum_{i=1}^{n}\alpha_{i}f_{p_i} \Bigg\vert_*&\geq \Bigg\vert \sum_{i=1}^{n}\alpha_{i} \int_{\Delta} f_{p_i} d\rho_{I} \Bigg\vert\\
&=\Bigg\vert \sum_{i=1}^{n}\alpha_{i} \int_{\Delta} f_{p_i} d\Big(\nu_{s_i,I\cap N_{s_i} }^{(p_i)}\otimes \mu_i\Big) \Bigg\vert\\
&= \Bigg\vert \sum_{i=1}^{n}\alpha_i \int_{\{0,1\}^{\N\setminus \{p_i\}}}\int_{\{0,1\}_{p_i}}f_{p_i} d\nu_{s_i,I\cap N_{s_i} }^{(p_i)}d\mu_i \Bigg\vert\\
&=\Bigg\vert \sum_{\alpha_{i}>0}\alpha_i \int_{\{0,1\}^{\N\setminus \{p_i\}}}\int_{\{0,1\}_{p_i}}f_{p_i} d\Big(\frac1{s_i}\delta^{(p_i)}_0+\frac{s_i-1}{s_i}\delta^{(p_i)}_1\Big)d\mu_i\\
&+\sum_{\alpha_{j}<0}\alpha_j \int_{\{0,1\}^{\N\setminus \{p_j\}}}\int_{\{0,1\}_{p_j}}f_{p_j} d\Big(\frac{s_j-1}{s_j}\delta^{(p_j)}_0+\frac1{s_j}\delta^{(p_j)}_1\Big)d\mu_j \Bigg\vert\\
&=\Bigg\vert \sum_{\alpha_{i}>0}\alpha_i \Big(\frac{s_i-1}{s_i}-\frac1{s_i}) +\sum_{\alpha_{j}<0}\alpha_j (\frac1{s_j}-\frac{s_j-1}{s_j}\Big) \Bigg\vert
\end{align*}
\begin{align*}
&=\sum_{\alpha_{i}>0}|\alpha_i| \Big(\frac{s_i-2}{s_i}\Big)+\sum_{\alpha_{j}<0}|\alpha_i| \Big(\frac{s_j-2}{s_j}\Big)\\
&\geq \frac{1}{3}\sum_{i=1}^{n}|\alpha_{i}|.
\end{align*}
\end{proof}

\section{Main results}
Let us summarize the key properties proved in the above sections.\begin{enumerate}
\item[(1)] (Proposition \ref{prop: renorm subspace SD2P}) There is a (non closed) subspace $\mathcal{Y}$ of $C(\Delta)^*$ and a dual norm $\vert\cdot\vert$ in $\mathcal{Y}$ such that $\Vert \cdot\Vert_{C(\Delta)^*}\leq \vert\cdot\vert$ on $\mathcal{Y}$. Furthermore, the restriction to $\mathcal{Y}$ of every element of $C(\Delta)$ is an element in $\mathcal{Y}_*$, the predual of $(\mathcal{Y},\vert\cdot\vert)$.
\item[(2)] (Lemmas \ref{lemma: f_p is a good set} and \ref{lemma: isormphic copy of ell_1}) There is an isometric $\ell_1$-sequence $\{f_p\}$ in $C(\Delta)$ such that $\{{f_{p}}_{\mid \mathcal{Y}}\}\subset \mathcal{Y}_*$ is an octahedral set for $(\mathcal{Y}_*,|\cdot|_*)^{**}$ and also it is an isomorphic $\ell_1$-sequence.
 \end{enumerate}

The results of the above sections allow us to get now the bidual octahedral  renorming for $C(\Delta)$.
\begin{theorem}\label{thm: renorming C(Delta)*}
	The Banach space $C(\Delta)$ admits an equivalent norm such that its bidual norm is octahedral.
\end{theorem}
\begin{proof}
Denote by $j_\mathcal{Y}$ the natural bounded linear map from $C(\Delta)$ to $\left.j(C(\Delta))\right|_\mathcal{Y}:=\{x_{\mid \mathcal{Y}}:x\in C(\Delta)\subset C(\Delta)^{**}\}$ and by $i$ the inclusion map from $\left.j(C(\Delta))\right|_\mathcal{Y}$ to $\mathcal{Y}_*$. Hence, $S:=i\circ j_\mathcal{Y}$ is a bounded linear map from $C(\Delta)$ to $\mathcal{Y}_*$.   Let  $Z$ (respectively, $Z_*$) be the subspace given by the closed linear span of $f_p$ in $C(\Delta)$ (respectively, by $\{{f_p}_{\mid \mathcal{Y}}\}$ in $(\mathcal{Y}_*,|\cdot|_*)$). 

By Lemma \ref{lemma: isormphic copy of ell_1},   $\left.S\right|_Z$ is an onto isomorphism from $Z\subset C(\Delta)$ onto $(Z_*, \vert \cdot \vert_*)$, and $\{{f_p}_{\mid \mathcal{Y}}\}$ is an octahedral set for $(\mathcal{Y}_*,|\cdot|_*)^{**}$ by Lemma \ref{lemma: f_p is a good set}. Now, applying Proposition \ref{prop: abstract bidual renorming} we are done.
\end{proof} 

As a consequence, we get the main announced result, which answers the Question \ref{question: question1}  in the separable case.

\begin{theorem}\label{thm: separable bidual octahedral}
	If $X$ is a separable Banach space containing a subspace isomorphic to $\ell_1$, then there is an equivalent norm in $X$ such that the bidual $X^{\ast\ast}$ is octahedral. 
\end{theorem}
\begin{proof}
Assume that $X$ contains a subspace isometric to $\ell_1$. From \cite[Theorem~2]{dilworth_dual_2000} (see also \cite{hagler_banach_1973}) we know that there is a closed subspace $Y$ of $X$ such that $C(\Delta)$ is isometric to $X/Y$. Denote by $\pi$ the quotient map from $X$ to $X/Y$. By the proof of Theorem~\ref{thm: renorming C(Delta)*} there is an equivalent norm $|\cdot|$ on $X/Y$ such that its bidual norm is octahedral, moreover there exists an octahedral set $\{w_p\colon p\in \N  \}\subset B_{(X/Y,|\cdot |)}$ whose linear span $W$ is isomorphic to $\ell_1$. 

Since $\pi(B^{\circ}_X)=B^{\circ}_{X/Y}$, then there is a bounded sequence $\{z_p\}\subset X$  such that $\pi(z_p)=w_p$ for every $p\in \N$. Denote by $Z$ the closed linear span of $\{z_p\}$. From the boundedness of $\pi$ and $\{z_p\}$ we conclude the existence of $K,L,M>0$ such that for every $n\in\N$ and $\alpha_1,\dots, \alpha_n\in \R$ we have that
\[
K\sum_{i=1}^{n}|\alpha_i|\leq \Big|\sum_{i=1}^{n}\alpha_i w_{p_i} \Big|\leq L\Big\|\sum_{i=1}^{n}\alpha_i z_{p_i} \Big\|\leq M\sum_{i=1}^{n}|\alpha_i|.
\]
Thus $\pi|_Z\colon Z\to (W,|\cdot |)$ is an onto isomorphism. Now, applying Proposition \ref{prop: abstract bidual renorming} we are done.
 \end{proof}

Note that the separability assumption is only used in the above proof, to assure that for every separable Banach space $X$ with $\ell_1$-copies, there is a linear and bounded operator from $X$ onto $C(\Delta)$. This last assertion is false in the non-separable case, for example it is known to be false if $X=\ell_{\infty}$, since every separable quotient of $\ell_{\infty}$ is reflexive \cite{hagler_banach_1973}. Then the above proof does not work in the general case. However,  our techniques can be applied to some non-separable Banach spaces. Indeed, assume that $X$ is a Banach space with the separable complementation property, that is, for every separable subspace $Y$ of $X$, there is a complemented separable subspace $Z$ of $X$ such that $Y\subset Z$. If moreover $X$ contains an isomorphic copy of $\ell_1$, then there is a complemented subspace $Z$ of $X$ containing $\ell_1$-copies. Hence the existence of a linear and continuous operator from $X$ onto $C([0,1])$ is clear, and the proof of Theorem \ref{thm: separable bidual octahedral} works. Then we have the following

\begin{corollary}\label{corollary:WCD} Let $X$ be a Banach space with the separable complementation property containing isomorphic copies of $\ell_1$. Then there is an equivalent norm in $X$ such that the bidual $X^{\ast\ast}$ is octahedral.
\end{corollary}

The above corollary can be applied, for example, to the family of weakly countably determined Banach spaces, since a member of the above family  satisfies the separable complementation property (see \cite[Chapter VI, Lemma 2.4]{DGZ}).   

Recall that a Banach space is said to be ($w^*$-)strongly regular if every closed, bounded and convex subset of $X$ contains convex combinations of ($w^*$-)slices with diameter arbitrarily small. It is known that $\Xast$ is strongly regular if and only if $X$ does not contain isomorphic copies of $\ell_1$, and also it is known that strong regularity and $w^*$-strong regularity are equivalent properties in dual Banach spaces \cite[Corollary~VI.18]{ghoussub_some_1987}.

The next consequence collects the different characterizations of octahedrality, strong diameter two property and strong regularity through the existence of isomorphic copies of $\ell_1$ and it is a dual answer to Question 2, in the setting of separable predual.
\begin{corollary}\label{corollary: strongregularity}
	Let $X$ be a separable Banach space. The following are equivalent:
	\begin{itemize}
		\item[(i)] $X$ contains a subspace isomorphic to $\ell_1$.
		\item[(ii)] $\Xast$ fails to be strongly regular.
		\item[(iii)] $X^*$ fails to be $w^*$-strongly regular.
		\item[(iv)] $X$ has an equivalent octahedral norm.
		\item[(v)] $X$ has an equivalent norm such that every convex combination of $w^*$-slices in $B_{X^*}$ has diameter two.
		\item[(vi)] For every $\varepsilon>0$ there is an equivalent norm in $X$ such that every convex combination of slices in $B_{X^*}$ has diameter, at least, $2-\varepsilon$.
		\item[(vii)] there exists an equivalent norm in $X$ such that $\Xast$ has the strong diameter two property.
		\item[(viii)] there exists an equivalent norm in $X$ such that $X^{\ast\ast}$ is octahedral.
	\end{itemize}
\end{corollary}
\begin{proof}The equivalence between the assertions (i) to (vi), and between (vii) and (viii), were written in \cite{becerra_octahedral_2014}. The equivalence between (i) and (viii) is the Theorem \ref{thm: separable bidual octahedral}.
\end{proof}

A dual Banach space $X^*$ is said to have the $w^*$-strong diameter two property ($w^*$-SD2P) if every convex combination of $w^*$-slices in $B_{X^*}$ has diameter two. Observe that the above result gives, in particular, that SD2P and $w^*$-SD2P in $\Xast$ are equivalent under renorming, whenever $X$ is separable. However, these two properties are not equivalent (see \cite{becerra_octahedral_2014}, see also \cite{lopez})). 

Note that Corollary~\ref{corollary: strongregularity} establishes a dichotomy: either every convex bounded subset of the dual space has
arbitrarily small convex combinations of $w^*$-slices, or there exists a dual unit ball such that every convex combination of weak slices has diameter 2.

Recall that the ball topology $b(X)$ on a Banach space $X$, is the coarsest topology on $X$ so that the norm closed balls of $X$ are $b(X)$-closed. If $X$ does not contain isomorphic copies of $\ell_1$, then for every equivalent norm in $X$, the ball topology of $X^{**}$ coincides with the weak-star topology on the bidual unit ball (see \cite{godefroy_kalton}). Another consequence of Theorem \ref{thm: separable bidual octahedral} is the following (see \cite{godefroy_kalton}).

\begin{corollary}\label{corollary: balltopology} Let $X$ be a separable Banach space and denote by $b(X^{**})_1$ the ball topology in $X^{**}$ restricted to $B_{X^{**}}$. If $X$ contains a subspace isomorphic to $\ell_1$, then there is an equivalent norm in $X$ so that, not only $b(X^{**})_1$ fails to be Hausdorff, but every pair of nonempty $b(X^{**})_1$-open subsets of $B_{X^{**}}$ has nonempty intersection.
\end{corollary}
\begin{proof} The fact that $b(X^{**})_1$ is not Hausdorff, under the hypotheses of the above corollary, is a consequence of Theorem 9.3 in \cite{godefroy_kalton}. The last conclusion in the above corollary is a consequence of Theorem \ref{thm: separable bidual octahedral} joint to the equivalence between (1) and (3) in Lemma 9.1 of \cite{godefroy_kalton}, which is valid in the non-separable case with the same proof.
\end{proof}

It is worth saying that in \cite[Question F]{godefroy_kalton} it is asked if $b(X^{**})_1$ fails to be Hausdorff for every equivalent norm in $X$, whenever $X$ contains a subspace isomorphic to $\ell_1$. 

We will end this note with a couple of questions for the non-separable case. We do not know whether Question \ref{question: question1} is separably determined nor whether it is possible to renorm $\ell_{\infty}$ such that its bidual is octahedral.

\section*{Acknowledgement}
We would like to thank Prof. G. Godefroy for his interest during the development of this work and notifying us about Corollary \ref{corollary: balltopology}. Also we thank an anonymous referee for helpful comments that improved the exposition and for notifying us about Corollary \ref{corollary:WCD}.


\begin{thebibliography}{1}
	 	\bibitem{becerra_octahedral_2014}
	J.~Becerra~Guerrero, G.~{L}\'{o}pez-{P}\'{e}rez, and A.~Rueda~Zoca,
	\emph{Octahedral norms and convex combination of slices in {Banach} spaces},
	J. Func. Anal. \textbf{266} (2014), 2424--2435.
	
	\bibitem{becerra_extreme_2015}
	J.~Becerra~Guerrero, G.~{L}\'{o}pez-{P}\'{e}rez, and A.~Rueda~Zoca,
	\emph{Extreme differences between weakly open subsets and convex combinations of slices in Banach spaces},
	Adv. Math. \textbf{269} (2015), 56--70.
	
	\bibitem{DGZ} R.~Deville, G.~Godefroy, and V.~Zizler, \emph{{S}moothness and renormings in
		{B}anach spaces}, Pitman Monographs and Surveys in Pure and Applied
	Mathematics, vol.~64, Longman Scientific \& Technical, Harlow, 1993.
	
	\bibitem{dilworth_dual_2000}
	S. J. Dilworth, M. Girardi, and J. Hagler, \emph{Dual Banach spaces which contain isometric copy of $L_1$}, Bull. Pol. Acad. Sci. Math. \textbf{48} (2000), 1--12.
	
	\bibitem{ghoussub_some_1987}
	N.~{G}houssoub, G.~{G}odefroy, B.~{M}aurey, and W.~{S}chachermayer,
	\emph{{S}ome topological and geometrical structures in {B}anach spaces}, Mem. Amer. Math. Soc. \textbf{378} (1987).
	
	\bibitem{godefroy_metric_1989}
	G.~Godefroy, \emph{{M}etric characterization of first {B}aire class linear forms and octahedral norms}, Studia Math. \textbf{95} (1989), 1--15.
	
	\bibitem{godefroy_kalton}
	G.~Godefroy and N.~J.~Kalton, \emph{The ball topology and its applications}, Contemporary Math. \textbf{85} (1989), 195--237.
	
	\bibitem{godefroy_maurey} G.~Godefroy and B.~Maurey. \emph{Normes lisses et anguleuses sur les espaces de Banach s\'eparables}, unpublished preprint.
	
	\bibitem{hagler_banach_1973}
	J.~Hagler and C.~Stegall, \emph{Banach spaces whose duals contain complemented subspaces isomorphic to $C[0,1]^{\ast}$}, J. Func. Anal. \textbf{13} (1973), 233--251.
	
	\bibitem{haller} R.~Haller, J.~Langemets, and M.~P\~oldvere. \emph{On duality of diameter 2 properties}, J. Conv. Anal. \textbf{22} (2015), no. 2, 465--483.
	
	\bibitem{kadets}
	V.~Kadets, V.~Shepelska, and D.~Werner, \emph{Thickness of the unit sphere, $\ell_1$-types, and the almost Daugavet property}, Houston J. Math. \textbf{37} (2011), 867--878.
	
	\bibitem{kadets2}
	V.~Kadets, R.~Shvidkoy, G.~Sirotkin, and D.~Werner, \emph{Banach spaces with the Daugavet property}, Trans. Amer. Math. Soc. \textbf{352} (2000), 855--873.
	
	\bibitem{kaijser} S.~Kaijser, \emph{A note on dual Banach spaces}, Math. Scand. \textbf{41} (1977), 325--330.
	
	\bibitem{lopez}
	G.~{L}\'{o}pez-{P}\'{e}rez, M.~ Mart\'{i}n, and A.~Rueda~Zoca,
	\emph{Strong diameter two property and convex combination of slices reaching the unit sphere}, to appear in Mediterr. J. Maths., preprint at arXiv:1703.04749.
	
	\bibitem{maurey}
	B.~Maurey, \emph{Types and $\ell_1$-subspaces}, Longhorn Notes, Texas Functional Analysis Seminar, Austin, Texas 1982/1983.
	
	\bibitem{rosenthal} H.~Rosenthal, \emph{A characterization of Banach spaces containing} $l^1$, Proc. Nat. Acad. Sci. U.S.A., \textbf{71} (1974), 2411--2413. 
	
\bibitem{schachermayer_moduli_1989}
W.~Schachermayer, A.~Sersouri, and E.~Werner,
\emph{Moduli of nondentability and the Radon--Nikod\'{y}m property in Banach spaces},
Israel J. Math. \textbf{65} (1989), 225--257.

\bibitem{yagoub-zidi}
Y. Yagoub-Zidi, \emph{Some isometric properties of subspaces of function spaces}, Mediterr. J. Maths. \textbf{10} (2013), 4, 1905--1915.

\end{thebibliography}
\end{document}